\documentclass[11pt,reqno]{amsart}
\usepackage{amsmath, amssymb, amsthm, bm, mathtools, enumitem, caption}
\usepackage{hyperref}
\usepackage{fullpage}

\usepackage[noabbrev,capitalize]{cleveref}
\usepackage{adjustbox}
\crefname{equation}{}{}
\usepackage{graphicx, tikz}
\usepackage{microtype}

\newtheorem{theorem}{Theorem}[section]

\newtheorem{corollary}[theorem]{Corollary}
\newtheorem{proposition}[theorem]{Proposition}

\newtheorem*{conjecture*}{Conjecture}

\theoremstyle{definition}

\theoremstyle{remark}

\newtheorem*{example}{Example}

\numberwithin{equation}{section}

\newcommand{\N}{\mathbb N}
\newcommand{\T}{\mathcal{T}}
\newcommand{\PGen}{\mathcal{P}}

\DeclareMathOperator{\Tr}{Tr}

\newcommand{\leg}[2]{\genfrac{(}{)}{}{}{#1}{#2}}

\newcommand{\HH}{\mathbb H}

\newcommand{\C}{\mathbb C}
\newcommand{\SL}{\mathrm{SL}}
\newcommand{\F}{\mathbb F}

\newcommand{\Z}{\mathbb Z}

\newcommand{\lp}{\left(}
\newcommand{\rp}{\right)}

\renewcommand{\pmod}[1]{\ \left( \mathrm{mod} \, #1 \right)}
\renewcommand{\Re}{\mathrm{Re}}

\title[Ramanujan's partition generating functions modulo $\ell$]{Ramanujan's partition generating functions modulo $\ell$}

\dedicatory{In honor of founding Editor-in-Chief Krishnaswami Alladi}

\author{Kathrin Bringmann, William Craig \and Ken Ono}
\address{Department of Mathematics and Computer Science\\Division of Mathematics\\University of Cologne\\ Weyertal 86-90 \\ 50931 Cologne \\Germany}
\email{kbringma@math.uni-koeln.de}

\address{Department of Mathematics, United States Naval Academy, 572C Holloway Road Mail Stop 9E.
	Annapolis, MD 21402}
\email{wcraig@usna.edu}

\address{Dept. of Mathematics, University of Virginia, Charlottesville, VA 22904}
\email{ken.ono691@virginia.edu}

\allowdisplaybreaks

\begin{document}
	\begin{abstract} For the partition function $p(n)$, Ramanujan proved the striking identities
		\begin{displaymath}
			\begin{split}
				\PGen_5(q):=\sum_{n\ge0} p(5n+4)q^n &=5\prod_{n\ge1} \frac{\left(q^5;q^5\right)_{\infty}^5}{(q;q)_{\infty}^6},\\
				\PGen_{7}(q):=\sum_{n\ge0} p(7n+5)q^n &=7\prod_{n\ge1}\frac{\left(q^7;q^7\right)_{\infty}^3}{(q;q)_{\infty}^4}+49q
				\prod_{n\ge1}\frac{\left(q^7;q^7\right)_{\infty}^7}{(q;q)_{\infty}^8},
			\end{split}
		\end{displaymath}
		where $(q;q)_{\infty}:=\prod_{n\ge1}(1-q^n).$ As these identities
		imply his celebrated congruences modulo 5 and 7, it is natural to seek, for primes $\ell \geq 5,$ closed form expressions of the power series
		$$
		\PGen_{\ell}(q):=\sum_{n\ge0} p(\ell n-\delta_{\ell})q^n\pmod{\ell},
		$$
		where $\delta_{\ell}:=\frac{\ell^2-1}{24}.$  In this paper, we prove that
		$$
		\PGen_{\ell}(q)\equiv c_{\ell} \dfrac{{\mathcal T}_{\ell}(q)}{\lp q^\ell; q^\ell \rp_\infty} \pmod{\ell},
		$$
		where $c_{\ell}\in \Z$ is explicit  and ${\T}_{\ell}(q)$ is the generating function for the Hecke traces of $\ell$-ramified values of special Dirichlet series for weight $\ell-1$ cusp forms on $\SL_2(\Z)$. This is a new proof of Ramanujan's congruences modulo 5, 7, and 11, as there are no nontrivial cusp forms of weight 4, 6, and 10.
	\end{abstract}

	\maketitle

	\section{Introduction and Statement of Results}
	A {\it partition} of $n$ is any nonincreasing sequence of positive integers that sum to $n.$  The number of partitions of $n$ is denoted $p(n)$ (by convention, we let $p(0):=1$ and $p(n):=0$ for $n<0$).  Ramanujan famously proved (see \cite{BerndtOno, Ramanujan1919}), for every non-negative integer $n$, that
	\begin{displaymath}
		\begin{split}
			p(5n+4)&\equiv 0\pmod 5,\\
			p(7n+5)&\equiv 0\pmod 7,\\
			p(11n+6)&\equiv 0\pmod{11}.
		\end{split}
	\end{displaymath}
	For the congruences with modulus 5 and 7, he used the beautiful identities
	\begin{displaymath}
		\begin{split}
			\PGen_5(q):=\sum_{n\ge0} p(5n+4)q^n &=5\prod_{n\ge1} \frac{\left(q^5;q^5\right)_{\infty}^5}{(q;q)_{\infty}^6},\\
			\PGen_{7}(q):=\sum_{n\ge0} p(7n+5)q^n &=7\prod_{n\ge1}\frac{\left(q^7;q^7\right)_{\infty}^3}{(q;q)_{\infty}^4}+49q
			\prod_{n\ge1}\frac{\left(q^7;q^7\right)_{\infty}^7}{(q;q)_{\infty}^8},
		\end{split}
	\end{displaymath}
	where $(q;q)_{\infty}:=\prod_{n\ge1}(1-q^n).$
	In 1969, with the help of binary theta functions, Winquist \cite{Wi} was able to offer another identity that proved Ramanujan's congruence with modulus 11.

	In the spirit of these identities, for every prime $\ell\geq 5,$ we determine the $q$-series $\PGen_{\ell}(q)\in \F_{\ell}[[q]]$
	\begin{equation*}
		\PGen_{\ell}(q):=\sum_{n\ge0} p(\ell n-\delta_{\ell})q^n\pmod{\ell},
	\end{equation*}
	where $\delta_{\ell}:=\frac{\ell^2-1}{24}.$ These expressions involve the generating functions of ``weighted Hecke traces'' of special values of specific Dirichlet series associated to weight $\ell-1$ Hecke eigenforms on $\SL_2(\Z)$ (for background see \cite{HMFBook} or \cite{CBMS}).

	To define these Hecke traces, first suppose that ($q:=e^{2\pi i z}$ throughout)
	$$f(z):=q+\sum\limits_{n\geq 2}a_f(n)q^n\in S_{2k}$$ 
	is an even integer weight $2k$ Hecke eigenform on $\SL_2(\Z)$.
	For $s\in\C$ with $\Re(s)>2k,$ the {\it twisted quadratic Dirichlet series} is defined by
	\begin{equation*}
		D(f;s):=\sum\limits_{n\geq 1}\frac{\leg{12}{n} a_f\left(\frac{n^2-1}{24}\right)}{n^s},
	\end{equation*}
	where $\leg{\cdot}{\cdot}$ denotes the Kronecker symbol. Note that we set $a_f(n) := 0$ if $n\notin \Z$. Furthermore, if $k\geq 2$, $0\leq j\leq k-2$, and $m\geq 0,$ then we let
	\begin{align*}
		\beta(k,j,m):=\frac{(-1)^{j+1}\Gamma\!\left(k-\frac{1}{2}\right)\Gamma\!\left(k+\frac{1}{2}\right)}{9}\left(\frac{6}{\pi}\right)^{2k}\frac{(2k+m-2)!(k-j-1)^{[k]}\left(\frac{3}{2}\right)^{\![j]}}{j! m! (2k-j-2)!\left(-\frac{1}{2}-j\right)^{\![k]}\left(\frac{5}{2}\right)^{\![j]}},
	\end{align*}
	where $\Gamma(\cdot)$ is the usual Gamma-function. Moreover the {\it rising factorial} is given by
	\begin{equation*}
		(x)^{[j]}:=\begin{cases} x(x+1)\cdots (x+j-1)  \ \ \ \ \ &{\text {\rm if}}\ j\geq1,\\ 
			1  \ \ \ \ \ &{\text {\rm if } j=0},
		\end{cases}
	\end{equation*}
	which are companions of the usual {\it falling factorials}
	\begin{equation*}
		(x)_m:=\begin{cases}  x(x-1)\cdots (x-m+1) \ \ \ \ &{\text {\rm if $m\geq 1$}},\\
			1 \ \ \ \ \ &{\text {\rm if $m=0$}},\\
			\frac{1}{(x)_{-m}} \ \ \ \ \ &{\text {\rm if $m\leq -1$}}.
		\end{cases}
	\end{equation*}
	For such $f\in S_{2k},$  we define the following sums of values of Dirichlet series by\footnote{Convergence can be concluded from Theorem 1.4 of \cite{GOSS}.}
	\begin{equation*}
		D_f:= \sum\limits_{j=0}^{k-2}\sum_{m\ge0} \beta(k ,j,m) D(f;2k+1+2m+2j).
	\end{equation*}
	Moreover we define, for $n\in \N$, the {\it weight $2k$ Hecke trace} by
	\begin{equation*}
		\Tr_{2k}(n):=\sum_{f}  a_f(n) \frac{D_f}{||f||},
	\end{equation*}
	where the sum runs over the normalized Hecke eigenforms $f\in S_{2k},$ and the Petersson norms of $f$, $||f||$, is defined as ($z=x+iy$ throughout)
	$$
	||f||:=\int_{\SL_2(\Z)\backslash \HH} |f(z)|^2 y^{2k}\frac{dxdy}{y^2}.
	$$
	As $a_f(n)$ is the eigenvalue of the Hecke operator $T_n$, we refer to the numbers $\Tr_{2k}(n)$ as Hecke traces.
	Finally, for primes $\ell \geq 5,$ 
	we collect the $\ell$-ramified values (i.e., the arguments that are multiples of $\ell$) if $2k=\ell-1$ as the Fourier coefficients of the generating function
	\begin{equation*}
		\T_{\ell}(q):=\sum_{n\ge1} \Tr_{\ell-1}(\ell n)q^n.
	\end{equation*}

	\begin{theorem}\label{Theorem1}
		If $\ell \geq 5$ is a prime, then
		$$
		\PGen_{\ell}(q)\equiv 
		c_{\ell}  \dfrac{\mathcal T_{\ell}(q)}{\lp q^\ell; q^\ell \rp_\infty} \pmod{\ell},
		$$
		where $c_{\ell}:= 2\cdot\overline3 (\frac{-1}{\ell}) (\frac{\ell+1}{2})!^{\ell-3} \pmod{\ell}$, where throughout $\overline a$ denotes the inverse of $a\pmod{\ell}$ and where $(\frac{\cdot}{\cdot})$ denotes the Kronecker symbol.
	\end{theorem}

	For $\ell\in \{5, 7, 11\}$, we have that $S_{\ell-1}=\{0\}.$ As there are no nontrivial cusp forms in these spaces, we immediately obtain a new proof of Ramanujan's famous partition congruences.

	\begin{corollary}\label{Corollary2}
		For $n\in \mathbb{N}$, we have
		\begin{displaymath}
			\begin{split}
				p(5n+4)&\equiv 0\pmod 5,\\
				p(7n+5)&\equiv 0\pmod 7,\\
				p(11n+6)&\equiv 0\pmod{11}.
			\end{split}
		\end{displaymath}
	\end{corollary}

	Moreover Theorem~\ref{Theorem1} immediately implies the following congruence formula for $p(\ell n-\delta_{\ell})\pmod \ell$ in terms of $p(0), p(1),\dots, p(n-1)$.

	\begin{corollary}\label{Corollary3} If $\ell\geq 5$ is a prime and $n\in \mathbb{N}$, then we have
		$$
		p(\ell n-\delta_{\ell}) \equiv c_{\ell} \sum_{\substack{j,m \geq 0 \\ \ell j+m=n}} p(j) \Tr_{\ell-1}(\ell m)\pmod \ell.
		$$
	\end{corollary}

	\begin{example} For the prime $\ell=13,$
		Theorem 1.4 and Corollary 1.3 of \cite{GOSS} gives
		\begin{align*}
			\T_{13}(q) &= -\dfrac{33108590592}{691} \Delta|U_{13}(z) \equiv 7 \Delta|U_{13}(z) \pmod{13},
		\end{align*}
		where $f|U_j(z):=\sum_{n\ge1} a_f(jn)q^n$ for $j\in\N$. Using $c_{13} \equiv 6 \pmod{13}$, we obtain
		\begin{displaymath}
			\begin{split}
				c_{13} \dfrac{\mathcal T_{13}(q)}{\lp q^{13}; q^{13} \rp_\infty} \equiv \dfrac{3\Delta|U_{13}(z)}{\lp q^{13}; q^{13} \rp_\infty}
				\equiv 11q+9q^2+3q^3+6q^4+12q^5+6q^6+q^8+\dots\pmod{13}.
			\end{split}
		\end{displaymath}
		To illustrate Theorem~\ref{Theorem1}, we note that
		\begin{displaymath}
			\begin{split}
				\PGen_{13}(q)&=\sum_{n \geq 1} p(13n-7)q^n=11q+490q^2+8349q^3+89134q^4+715220q^5+\dots\\
				&\equiv 11q+9q^2+3q^3+6q^4+12q^5+6q^6+q^8+\dots\pmod{13}.
			\end{split}
		\end{displaymath}
		Furthermore, Corollary~\ref{Corollary3} implies, for $n\in \mathbb{N}$, that
		\begin{displaymath}
			\begin{split}
				p(13n-7)&\equiv 3 \sum_{\substack{j,m \geq 0 \\ 13j+m=n}} p(j) \tau(13m)\pmod{13}.\\
			\end{split}
		\end{displaymath}
	\end{example}

	To obtain Theorem~\ref{Theorem1}, we make use of recent work of Gomez, the third author, Saad, and Singh \cite{GOSS} that offers an infinite family of generalizations of Euler's ``Pentagonal Number'' recurrence for $p(n)$. In Section 2 we recall these formulas, and in Section~\ref{Proofs} we use them to obtain Theorem~\ref{Theorem1}.

	\section*{Acknowledgements} The first author is funded by the European Research
	Council (ERC) under the European Union's Horizon 2020 research and innovation programme
	(grant agreement No. 101001179), the SFB/TRR 191
	``Symplectic Structure in Geometry, Algebra and Dynamics'', and  the DFG (Projekt-
	nummer 281071066 TRR 191).
	The third author is grateful for the support of the Thomas Jefferson Fund, the NSF
	(DMS-2002265 and DMS-2055118), and the Simons Foundation (SFI-MPS-TSM-00013279). The views expressed in this article are those of the authors and do not reflect the official policy or position of the U.S. Naval Academy, Department of the Navy, the Department of Defense, or the U.S. Government. We thank the referee for helpful comments on our paper.

	\section{Generalizations of Euler's ``Pentagonal number'' recurrence}

	For $n\in \mathbb{N},$ Euler's famous recurrence relation asserts that
	(see p. 12 of \cite{Andrews})
	\begin{align}\label{EulerRecurrence}
		p(n)=p(n-1)+p(n-2)-p(n-5)-p(n-7)+ \dots
		= \sum_{m\in \Z \setminus \{0\}} (-1)^{m+1}p(n-\omega(m)),
	\end{align}
	where $\omega(m):=\frac{3m^2+m}{2}$ is the $m$-th {\it pentagonal number}. This recurrence is one of the most efficient methods for computing partition numbers.

	Gomez, the third author, Saad, and Singh \cite{GOSS} proved that Euler's recurrence is the first case of an infinite family of rich recurrence relations satisfied by the partition numbers. To make this precise, we make use of {\it Dedekind's eta-function}
	\begin{equation*}
		\eta(z):=q^{\frac{1}{24}}\prod_{n\ge1}\left(1-q^n\right)=\sum_{n\in \Z} (-1)^n q^{\frac{1}{24}(6n+1)^2},
	\end{equation*}
	where $z\in \HH,$ the upper half of the complex plane.
	To define these relations, we require the differential operator $D :=\frac{1}{2\pi i}\frac{d}{dz} = q\frac{d}{dq}.$
	For $k\in \mathbb{N}_0,$ we define\footnote{To avoid confusing notation, we note that $R_{k}(z)$ is denoted $P_{k}(z)$ in \cite{GOSS}.} \footnote{We note a small typographical error in \cite{GOSS} (there the $(-1)^{r+1}$ is $(-1)^r$).}
	\begin{equation*}
		R_{k}(z) := \frac{(2k-1)(2k-2)_{k-1}^2}{2^{2k-2}} \sum_{\substack{r,s \geq 0 \\ r + s = k}} (-1)^{r+1} \frac{2s-1}{(2r)! (2s)!} D^r\left(\frac{1}{\eta(z)}\right) D^s(\eta(z)).
	\end{equation*}
	By \cite{GOSS}, we have
	\begin{align*}
		R_{k}(z)= \sum_{\substack{n\geq 0\\ m\in \Z}} (-1)^{m+1} g_{k}(n,m) p(n-\omega(m))q^n,
	\end{align*}
	where
	\begin{equation*}
		g_k(n,m):=\frac{(2k-1)(2k-2)_{k-1}^2}{2^{2k-2}}
		\sum_{r=0}^{k} (-1)^{k+r}\frac{2k-2r-1}{ (2r)! (2k-2r)!} (6m+1)^{2r}\left(24n - (6m+1)^2\right)^{k - r}.
	\end{equation*} 
	By Theorem 1.1 of \cite{GOSS}, for each $k \geq 0$, $R_{k}$ is a weight $2k$ holomorphic modular form on $\SL_2(\Z)$.
	These expressions are simple to compute for $k \leq 13$ apart from $k=12.$ Namely, Corollaries 1.2 and 1.3 of \cite{GOSS}  give the following identities in terms of
	the usual Eisenstein series
	\begin{equation*}
		E_{2k}(z) := 1 - \dfrac{4k}{B_{2k}} \sum_{n \geq 1} \sigma_{2k-1}(n) q^n,
	\end{equation*}
	where $B_r$ denotes the $r$-th Bernoulli number, $\sigma_r(n):=\sum_{d\mid n} d^r$ the \emph{$r$-th divisor sum}, and $\Delta(z):=\eta^{24}(z)$.

	\begin{theorem}\label{theorem2} 
		The following are true:

		\begin{enumerate}[leftmargin=*]
			\item[\rm(1)] If $k\in \{0, 1\},$ then we have
			$$
			R_{k}(z) = \begin{cases} -1 \ \ \ \ \ &{\text {\rm if }} k=0,\\
				0 \ \ \ \ \ &{\text {\rm if }} k=1.
			\end{cases}
			$$

			\item[\rm(2)] If $k \in \{2, 3, 4, 5, 7\},$ then we have
			$$
			R_{k}(z) =  \binom{2k-2}{k-2} E_{2k}(z).
			$$

			\item[\rm(3)] If $k \in \{6, 8, 9, 10, 11, 13\},$  then we have
			$$
			R_{k}(z) = \binom{2k-2}{k-2} E_{2k}(z) + \beta_{k}\Delta_{2k}(z),
			$$
			where
			\begin{equation*}
				\Delta_{2k}(z):=q+\sum_{n\ge2} \tau_{2k}(n)q^n:= \begin{cases}
					\Delta(z) \ \ \ \ \ &{\text  {\rm if $k=6,$}}\\
					\Delta(z)E_4(z) \ \ \ \ &{\text {\rm  if $k=8$}},\\
					\Delta(z)E_6(z) \ \ \ \ &{\text {\rm  if $k=9$,}}\\
					\Delta(z)E_4^2(z) \ \ \ \ \ &{\text {\rm  if $k=10$,}}\\
					\Delta(z)E_4(z)E_6(z) \ \ \ \ \ &{\text {\rm  if $k=11$,}}\\
					\Delta(z)E_4^2(z) E_6(z)\ \ \ \ \ &{\text {\rm  if $k=13,$}}
				\end{cases}
			\end{equation*}
			where we let
			\begin{displaymath}
				\beta_{k}:= \begin{cases} -\frac{33108590592}{691} \ \ \ \ \ &{\text {if $k=6,$}}\\ \ \\
					-\frac{187167592415232}{3617} \ \ \ \ \ &{\text  {if $k=8,$}}\\ \ \\
					-\frac{28682634201661440}{43867} \ \ \ \ \ &{\text {if $k=9,$}}\\ \ \\
					-\frac{8294726176465158144}{174611} \ \ \ \ \ &{\text {if $k=10,$}}\\ \ \\
					-\frac{101475065073734516736}{77683} \ \ \ \ \ &{\text {if $k=11,$}}\\ \ \\
					-\frac{1195065734266339700244480}{657931} \ \ \ \ \ &{\text {if $k=13.$}}
				\end{cases}
			\end{displaymath} 
			
		\end{enumerate}
	\end{theorem}

	Finally, for general $k,$ Theorem 1.4 of \cite{GOSS} gives the following expressions that make use of the weighted Hecke trace generating function 
	\begin{equation*}
		{T}_{2k}(z):=\sum_{n\ge1} \Tr_{2k}(n)q^n\in S_{2k}.
	\end{equation*}

	\begin{theorem}\label{theorem3} If $k \geq 6,$ with $k\neq 7$, then we have
		$$
		R_{k}(z)= \binom{2k-2}{k-2} E_{2k}(z) + {T}_{2k}(z).
		$$
	\end{theorem}

	These results are equivalent to the infinite family of recurrence relations given in the following corollary.

	\begin{corollary} \label{corollary3}
		If $n$ is a positive integer, then the following are true:
		\begin{enumerate}[leftmargin=*]
			\item[\rm(1)]
			We have that
			\begin{align*}
				p(n) = \sum_{m \in \Z\backslash\{0\}} (-1)^{m+1} p\lp n - \omega(m) \rp.
			\end{align*}

			\item[\rm(2)]
			If $k \in \{2, 3, 4, 5, 7\}$, then we have
			\begin{equation*}
				\hspace{.5cm}p(n) =\frac{1}{g_{k}(n,0)}\left( -\frac{4k}{B_{2k}}    \binom{2k - 2}{k - 2}\sigma_{2k - 1}(n)  + \sum_{m \in \Z\setminus\{0\}} (-1)^{m+1} g_k(n,m)\,p(n - \omega(m))\right).
			\end{equation*}

			\item[\rm(3)]
			If $k \in \{ 6, 8, 9, 10, 11, 13 \}$, then we have
			\begin{equation*}
				\hspace{.5cm}p(n) =\frac{1}{g_{k}(n,0)}\left( -\frac{4k}{B_{2k}}    \binom{2k - 2}{k - 2}\sigma_{2k - 1}(n) +\beta_{k}\tau_{2k}(n) + \sum_{m \in \Z\setminus\{0\}} (-1)^{m+1} g_k(n,m)\,p(n - \omega(m))\right).
			\end{equation*}

			\item[\rm(4)] 
			If $k \geq 6$, with $k\neq 7,$ then we have
			\begin{equation*}
				\hspace{.5cm}p(n) = \dfrac{1}{g_k\lp n, 0 \rp} \left( -\dfrac{4k}{B_{2k}} \binom{2k-2}{k-2} \sigma_{2k-1}(n) + \mathrm{Tr}_{2k}(n) + \sum_{m \in \Z\backslash\{0\}} \lp-1\rp^{m+1} g_k\lp n,m \rp p\lp n - \omega(m) \rp \right).
			\end{equation*}
		\end{enumerate}
	\end{corollary}

	\section{Proof of Theorem~\ref{Theorem1}}\label{Proofs}

	The proof of Theorem~\ref{Theorem1} requires the following elementary proposition regarding the congruence properties of certain examples of Corollary~\ref{corollary3} (4). Namely, we obtain a pentagonal number recurrence modulo $\ell$ for the Hecke traces with argument $\ell n$, where the pentagonal numbers $\omega(m)$ are restricted to a fixed congruence class modulo $\ell$.

	\begin{proposition} \label{congruence}
		If $\ell \geq 5$ is prime and $n$ is a positive integer, then
		\begin{align*}
			\mathrm{Tr}_{\ell-1}\lp \ell n \rp \equiv -3\cdot\overline2 \left(\frac{\ell+1}{2}\right)!^2 \sum_{\substack{m \in \Z \\ 6m \equiv -1 \pmod{\ell}}} (-1)^{m+1} p\lp \ell n - \omega(m) \rp \pmod{\ell}.
		\end{align*}
	\end{proposition}


	\begin{proof}
		By Corollary \ref{corollary3} (4), we have, for $k \geq 2$, that
		\begin{equation*}
			p(n) = \dfrac{1}{g_k\lp n, 0 \rp} \left( -\dfrac{4k}{B_{2k}} \binom{2k-2}{k-2} \sigma_{2k-1}(n) + \mathrm{Tr}_{2k}(n) + \sum_{m \in \Z\backslash\{0\}} \lp-1\rp^{m+1} g_k\lp n,m \rp p\lp n - \omega(m) \rp \right).
		\end{equation*}
		By letting $k = \frac{\ell-1}{2}$, the von Stadt--Clausen Theorem (for example, see \cite[Theorem 3, pg. 233]{IrelandRosen})  implies that the denominator of the Bernoulli number $B_{\ell-1}$ is divisible by $\ell$, which in turn implies that the divisor function contribution above vanishing modulo $\ell$. By then letting $n \mapsto \ell n$, we obtain
		\begin{equation}\label{ef}
			p\lp \ell n \rp \equiv \dfrac{1}{g_{\frac{\ell-1}{2}}\lp\ell n, 0 \rp} \Bigg(\mathrm{Tr}_{\ell-1}\lp \ell n \rp + \sum_{m \in \Z\backslash\{0\}} \lp -1 \rp^{m+1} g_{\frac{\ell-1}{2}}\lp \ell n, m \rp p\lp \ell n - \omega(m) \rp \Bigg) \pmod{\ell}.
		\end{equation}
		By direct calculation, we have
		\begin{align}\label{calculation}
			&g_{\frac{\ell-1}{2}}\lp \ell n, m \rp = \dfrac{\lp \ell-2 \rp \lp \ell - 3 \rp_{\frac{\ell-3}{2}}^2}{2^{\ell-3}} \sum_{r = 0}^{\frac{\ell-1}{2}} \lp -1 \rp^{\frac{\ell-1}{2} + r} \dfrac{\ell - 2 - 2r}{\lp 2r \rp! \lp \ell - 1 - 2r \rp!} 
			\lp 6m+1 \rp^{2r} \lp 24\ell n - \lp 6m+1 \rp^2 \rp^{\frac{\ell-1}{2} - r} \nonumber\\
			&\equiv \dfrac{16}{2^\ell} \lp \ell-3 \rp_{\frac{\ell-3}{2}}^2 \lp 6m+1 \rp^{\ell-1} \sum_{r=0}^{\frac{\ell-1}{2}} \dfrac{2r+2}{\lp 2r \rp! \lp \ell - 1 - 2r \rp!}  
			\equiv \dfrac{32}{2^\ell} \lp \ell-3 \rp_{\frac{\ell-3}{2}}^2 \lp 6m+1 \rp^{\ell-1} \sum_{r=0}^{\frac{\ell-1}{2}} \binom{\ell\!-\!1}{2r} \dfrac{r+1}{\lp \ell-1 \rp!}  \nonumber\\ &\equiv \varrho_\ell \lp 6m+1 \rp^{\ell-1} \equiv
			\begin{cases}
				\varrho_\ell & m \not \equiv - \overline{6}, \\
				0 & m \equiv -\overline{6},
			\end{cases}
			\pmod{\ell},
		\end{align}
		where
		\begin{align} \label{eqn1}
			\varrho_\ell := \dfrac{32}{2^\ell} \lp \ell-3 \rp_{\frac{\ell-3}{2}}^2 \sum_{r=0}^{\frac{\ell-1}{2}} \binom{\ell-1}{2r} \dfrac{r+1}{\lp \ell-1 \rp!}.
		\end{align}
		To compute $\varrho_\ell$, we note that for $M \geq 1$, we have
		\begin{align*}
			\sum_{r=0}^M \binom{2M}{2r} r = 2^{2M-2}  M\ \ \ {\text {\rm and}} \ \ \ \sum_{r=0}^M \binom{2M}{2r} = 2^{2M-1}.
		\end{align*}
		Therefore, by setting $M = \frac{\ell-1}{2},$ we have
		\begin{align*}
			\sum_{r=0}^{\frac{\ell-1}{2}} \binom{\ell-1}{2r} (r+1) \equiv \dfrac{\ell-1}{2} 2^{\ell-3} + 2^{\ell-2} = 3\cdot 2^{\ell-4}\pmod{\ell}.
		\end{align*}
		Combining this with \eqref{eqn1}, we obtain
		\begin{align*}
			\varrho_\ell \equiv \dfrac{6 \lp \ell-3 \rp_{\frac{\ell-3}{2}}^2}{\lp \ell-1 \rp!}  \pmod{\ell}.
		\end{align*}
		After application of Wilson's Theorem we see that
		\begin{align*}
			\varrho_\ell \equiv -6 \lp \ell-3 \rp_{\frac{\ell-3}{2}}^2 \pmod{\ell}.
		\end{align*}
		Finally, we note that
		\begin{equation*}
			(\ell-3)_{\frac{\ell-3}{2}} \equiv \frac{(-1)^{\frac{\ell-3}{2}}\left(\frac{\ell+1}{2}\right)!}      {2}  \pmod{\ell}.
		\end{equation*}
		Thus
		\begin{align}\label{rl}
			\varrho_\ell \equiv -6 \lp \dfrac{(-1)^{\frac{\ell-3}{2}}\left(\frac{\ell+1}{2}\right)!}{2} \rp^2 \equiv -3\cdot\overline2 \left(\frac{\ell+1}{2}\right)!^2 \pmod{\ell}.
		\end{align}
		Therefore, we have by \eqref{ef} and \eqref{calculation}
		\begin{align} \label{eqn2}
			\varrho_\ell  p\lp \ell n \rp &\equiv \mathrm{Tr}_{\ell-1}\lp \ell n \rp + \varrho_\ell \sum_{m \in \Z\backslash\{0\}} \lp -1 \rp^{m+1} \lp 6m+1 \rp^{\ell-1} p\lp \ell n - \omega(m) \rp \notag \\ &\equiv \mathrm{Tr}_{\ell-1}\lp \ell n \rp + \varrho_\ell \sum_{\substack{m \in \Z\backslash\{0\} \\ 6m \not\equiv -1 \pmod{\ell}}} \lp -1 \rp^{m+1} p\lp \ell n - \omega(m) \rp \pmod{\ell}.
		\end{align}
		Now, substituting $n \mapsto \ell n$ in \eqref{EulerRecurrence} and multiplying by $\varrho_\ell$ on both sides gives
		\begin{align*}
			\varrho_\ell  p\lp \ell n \rp \equiv \varrho_\ell \sum_{m \in \Z\backslash\{0\}} (-1)^{m+1} p\lp \ell n - \omega(m) \rp q^n \pmod{\ell}. 
		\end{align*}
		By subtracting \eqref{eqn2} from this on both sides, we obtain
		\begin{align*}
			0 \equiv - \mathrm{Tr}_{\ell-1}\lp \ell n \rp + \varrho_\ell \sum_{\substack{m \in \Z \\ 6m \equiv -1 \pmod{\ell}}} \lp -1 \rp^{m+1} p\lp \ell n - \omega(m) \rp \pmod{\ell}.
		\end{align*}
		Solving for $\mathrm{Tr}_{\ell-1}\lp\ell n\rp$ and substituting \eqref{rl} gives the claim.
	\end{proof}

	We are now ready to prove Theorem \ref{Theorem1}.

	\begin{proof}[Proof of Theorem~\ref{Theorem1}]
		Proposition~\ref{congruence} is equivalent to the generating function congruence
		\begin{align*}
			\T_{\ell}(q) \equiv -3\cdot\overline2 \left(\frac{\ell+1}{2}\right)!^2 \sum_{n \geq 0} \sum_{\substack{m \in \Z \\ \omega(m) \equiv \delta_\ell \pmod{\ell}}} (-1)^{m+1} p\lp \ell n - \omega(m) \rp q^{n} \pmod{\ell},
		\end{align*}
		where we note that $6m \equiv -1 \pmod{\ell}$ is equivalent to $\omega(m) \equiv \delta_\ell \pmod{\ell}$. By taking a convolution product, we see that
		\begin{align*}
			\sum_{n \geq 0} \sum_{\substack{m \in \Z \\ \omega(m) \equiv \delta_\ell \pmod{\ell}}} (-1)^{m+1} p\lp \ell n - \omega(m) \rp q^{n} \equiv \PGen_\ell\lp q \rp \theta_\ell(q) \pmod{\ell},
		\end{align*}
		for some $q$-series
		\begin{align*}
			\theta_\ell(q) := \sum_{s \in \Z} (-1)^{y_\ell(s)} q^{w_\ell(s)}.
		\end{align*}
		
		We now turn to the explicit calculation of $\theta_{\ell}(q),$ which then completes the proof. To this end,
		we observe that the $n$-th Fourier coefficient of $\PGen_\ell\lp q \rp \theta_\ell(q)$ is
		\begin{align*}
			\sum_{\substack{m \in \Z \\ \omega(m) \equiv \delta_\ell \pmod{\ell}}}\!\!\!\!\!\!\! (-1)^{m+1} p\lp \ell n - \omega(m) \rp \equiv \sum_{s \in \Z} (-1)^{y_\ell(s)} p\lp \ell n - \lp \ell w_\ell(s) + \delta_\ell \rp \rp \pmod{\ell}.
		\end{align*}
		To identify $w_\ell(s)$, we solve $\ell w_\ell(s) + \delta_\ell = \omega(m)$ for $m \equiv -\overline6 \pmod{\ell}$. Now, define $\alpha_\ell$ by $6\alpha_\ell = \ell m_\ell - 1$ with $m_\ell = \pm 1$ chosen so that $\alpha_\ell = \frac{\ell m_\ell - 1}{6} \in \Z$. Then by setting $m = \ell s + \alpha_\ell$ in the formula for $\omega(m)$ and simplifying, we see that 
		\begin{align*}
			\omega(\ell s + \alpha_\ell) = \ell \dfrac{3\ell s^2 + 6\alpha_\ell s + s}{2} + \dfrac{3\alpha_\ell^2 + \alpha_\ell}{2} = \ell \dfrac{3\ell s^2 + \ell m_\ell s}{2} + \delta_\ell.
		\end{align*}
		Thus
		\begin{align*}
			w_\ell(s) = \dfrac{3\ell s^2 + \ell m_\ell s}{2} = \begin{cases}
				\dfrac{3\ell s^2 + \ell s}{2} & \text{if } \ell \equiv 1 \pmod{6}, \\[+0.5cm]
				\dfrac{3\ell s^2 - \ell s}{2} & \text{if } \ell \equiv 5 \pmod{6}.
			\end{cases}
		\end{align*}
		Likewise, by comparing $(-1)^{y_\ell(s)} = (-1)^{m+1}$ if $m = \ell s + \alpha_\ell$ with the same choice of $\alpha_\ell$, we can set $y_\ell(s) = s + \alpha_\ell + 1$. We therefore obtain after some calculation that for $\ell \equiv 1 \pmod{6}$, we have, using \eqref{ef},
		\begin{align*}
			\theta_\ell(q)&=\sum_{s \in \Z} \lp -1 \rp^{s+\frac{\ell-1}{6}+1} q^{\frac{3s^2 + s}{2} \ell} = \lp -1 \rp^{\frac{\ell-1}{6}+1} \sum_{s \in \Z} \lp -1 \rp^s q^{\frac{3s^2+s}{2} \ell} 
			= \lp -1 \rp^{\frac{\ell+5}{6}} \lp q^\ell;q^\ell \rp_\infty.
		\end{align*}
		Likewise for $\ell \equiv 5 \pmod{6}$ we have, using \eqref{ef},
		\begin{align*}
			\theta_\ell(q) &= \sum_{s \in \Z} \lp -1 \rp^{s + \frac{\ell+1}{6} + 1} q^{\frac{3s^2-s}{2} \ell} = \lp -1 \rp^{\frac{\ell+1}{6}} \sum_{s \in \Z} \lp -1 \rp^{s+1} q^{\frac{3s^2-s}{2} \ell} 
			= \lp -1 \rp^{\frac{\ell-1}{6}} \lp q^\ell;q^\ell \rp_\infty.
		\end{align*}
		Now note that
		\begin{equation*}
			-\left(\frac{-1}{\ell}\right) = \begin{cases}
				(-1)^{\frac{\ell+5}{6}}&\text{if }\ell\equiv1\pmod{6},\\
				(-1)^{\frac{\ell+1}{6}}&\text{if }\ell\equiv5\pmod{6}.
			\end{cases}
		\end{equation*}
		We conclude
		\begin{align*}
			c_\ell \equiv -\leg{-1}{\ell}  \overline{-3\cdot\overline2\lp\dfrac{\ell+1}{2}\rp!^2} \equiv 2\cdot\overline3 \leg{-1}{\ell} \lp\dfrac{\ell+1}{2}\rp!^{\ell-3} \pmod{\ell},
		\end{align*}
		which completes the proof.
	\end{proof}

\end{document}